\documentclass[12pt, reqno]{amsart}
\usepackage{amsmath, amsthm, amscd, amsfonts, amssymb, graphicx, color, mathrsfs}
\usepackage[bookmarksnumbered, colorlinks, plainpages]{hyperref}
\usepackage{cite}
\usepackage[all]{xy}
\usepackage{slashed}
\usepackage{tikz-cd}
\usepackage{mathabx}
\usepackage{tipa}
\usepackage{soul}
\usepackage{cancel}
\usepackage{ulem}
\usepackage{algorithm}
\usepackage{algpseudocode}
\usepackage{enumitem}
\usepackage{multirow}
\usepackage{float}
\usepackage{verbatim}
\usepackage{amssymb}
\usepackage{extpfeil}
\usepackage{caption}
\usepackage{array}
\usepackage{cellspace}

\usepackage{pstricks, pst-plot}
\usepackage{pictex,dcpic}
\usepackage{tikz}
\usepackage{tikz-cd}

\allowdisplaybreaks

\newlength\tindent
\makeatletter
\newcommand{\FBSMA}{%
  \begingroup
  \let\cite@adjust\@empty
  \cite[p.~1826]{FBSM}%
  \endgroup}
\makeatother

\makeatletter
\newcommand{\GGNTheoremA}{%
  \begingroup
  \let\cite@adjust\@empty
  \cite[Proposition~3.3]{GGN}%
  \endgroup}
\makeatother
\makeatletter
\newcommand{\FBSMTheoremOne}{%
  \begingroup
  \let\cite@adjust\@empty
  \cite[Proposition~3]{FBSM}%
  \endgroup}
\makeatother
\newtheorem{theorem}{Theorem}[section]

\newtheorem{proposition}[theorem]{Proposition}
\newtheorem{corollary}[theorem]{Corollary}

\theoremstyle{definition}

\newtheorem{example}[theorem]{Example}

\theoremstyle{remark}
\newtheorem{remark}[theorem]{Remark}
\numberwithin{equation}{section}

\begin{document}
\setcounter{page}{1}

\title[]{Invariant almost Hermitian structures on the maximal real flag manifold of type $\mathcal{A}_3$}

\author[A. P. C. Freitas]{Ana P. C. Freitas}
\address{Ana P. C. Freitas \endgraf 
Universidade Federal da Grande Dourados, Faculdade de Ciências Exatas e Tecnologia, Rodovia Dourados–Itahum, km 12, Unidade II, Cidade Universitária, CEP 79804-970, Dourados - MS, Brazil. \endgraf
{\it E-mail address:} {\rm anafreitas@ufgd.edu.br}
 }

\author[B. Grajales]{Brian Grajales}%\textsuperscript{*}}
%\thanks{\textsuperscript{*}Corresponding author: {\rm bdgtriana@uem.br}}
\address{Brian Grajales \endgraf
Universidade Estadual de Maringá, Departamento de Matem\'{a}tica, Avenida Colombo, 5790, Campus Universitário, CEP 87020-900, Maringá - PR, Brazil. 
\endgraf
  {\it E-mail address:} {\rm bdgtriana@uem.br}
  }

\author[A. R. Oliveira]{Ailton R. Oliveira}
\address{Ailton R. Oliveira \endgraf
Universidade Estadual de Mato Grosso do Sul, Rodovia Itahum, km 12 s/n, Cidade Universitária de Dourados, CEP 79804-970, Dourados - MS, Brazil.\endgraf
{\it E-mail address:} {\rm aroliveira@uems.br}
  }

\thanks{}
\begin{abstract}
We study the Gray-Hervella classes of invariant almost Hermitian structures on the maximal real flag manifold of type $\mathcal{A}_3$. We first obtain an explicit parametrization of all invariant almost Hermitian structures and establish necessary and sufficient conditions for an invariant almost Hermitian structure to belong to each Gray-Hervella class. For each nonempty class, we give an explicit parametrization of the corresponding set of structures, describe it as a smooth manifold and determine its dimension. 
\end{abstract}
\maketitle

%\tableofcontents
\allowdisplaybreaks
\section{Introduction}
Almost Hermitian structures provide a natural generalization of K\"ahler geometry in which the almost complex structure is not necessarily parallel with respect to the Levi-Civita connection. The obstruction to this parallelism is measured by the intrinsic torsion of the almost Hermitian structure, which carries the same information as the covariant derivative of the fundamental form. In particular, an almost Hermitian structure is K\"ahler precisely when its intrinsic torsion vanishes. Gray and Hervella showed in \cite{GH} that the space of possible intrinsic torsions decomposes into four irreducible $\mathrm{U}(n)$-modules $\mathcal{W}_j,\ j=1,2,3,4$ and that these modules yield sixteen classes, which describe the different possible deviations from K\"ahler geometry and provide a unified framework for studying distinguished geometries such as nearly K\"ahler, almost K\"ahler and Hermitian structures.\\

On a homogeneous space, every invariant tensor is determined by its value at the origin. Thus, the Gray-Hervella class of an invariant almost Hermitian structure can be determined from the isotropy representation and the Lie bracket. Over the past decades, several authors have studied the Gray-Hervella classes of invariant almost Hermitian structures on different families of homogeneous spaces. We mention some contributions related to the questions considered here. In real dimension six, Abbena, Garbiero and Salamon obtained a combinatorial description of the Gray-Hervella classes of left-invariant almost Hermitian structures on the Iwasawa manifold and other nilmanifolds \cite{AGS}. Alekseevsky, Kruglikov and Winther classified six-dimensional homogeneous almost complex structures with semisimple isotropy and determined which combinations of the four components of the intrinsic torsion can occur for compatible invariant almost Hermitian structures \cite{AKW}. For almost abelian Lie groups, Fino and Paradiso classified the six-dimensional Lie algebras admitting invariant balanced Hermitian structures \cite{FP}. Andrada and Tolcachier later gave algebraic conditions characterizing all Gray-Hervella classes of invariant almost Hermitian structures on almost abelian Lie groups and studied harmonic almost complex structures within these classes \cite{AT}. In the case of complex flag manifolds, San Martin and Negreiros characterized invariant $(1,2)$-symplectic structures and classified these structures on full flag manifolds \cite{SMN}. San Martin and Silva later determined which generalized flag manifolds admit invariant strict nearly K\"ahler structures \cite{SMS}, Alves and da Silva classified invariant $\mathcal{G}_1$-structures and invariant quasi-K\"ahler structures on these spaces \cite{AS}, and Grama and Oliveira studied invariant almost Hermitian structures on flag manifolds to provide examples of Kähler-like scalar curvature \cite{GO}. \\

For real flag manifolds, one may have isotropy representations with equivalent irreducible submodules. Consequently, invariant metrics need not be diagonal with respect to an irreducible decomposition. In this setting, the compatibility condition between an invariant metric and an invariant almost complex structure imposes additional relations among their parameters. This leads to a different treatment from the one used for complex flag manifolds. In this paper, we consider the maximal real flag manifold of type $\mathcal{A}_3,$
\[
\mathbb{F}=\operatorname{SO}(4)/\mathrm{S}\left(\mathrm{O}(1)\times\mathrm{O}(1)\times\mathrm{O}(1)\times\mathrm{O}(1)\right).
\]
This is a six-dimensional homogeneous space whose isotropy representation contains equivalent irreducible submodules. We give an explicit parametrization of all invariant almost Hermitian structures on $\mathbb{F}$ by 
\begin{equation}\label{eq:A3-W123}
(z_j,\sigma_j)_{j=1}^{3},\ z_j\in\mathbb{C},\ \sigma_j\in\mathbb{R}^*,\ \sigma_j\operatorname{Im}(z_j)>0.
\end{equation}
The parameter space is then a smooth manifold of real dimension $9$ with eight connected components, determined by the signs of $\sigma_1,\sigma_2$ and $\sigma_3.$ Using these parameters, we obtain necessary and sufficient conditions for an invariant almost Hermitian structure to belong to each Gray-Hervella class, as well as explicit parametrizations of the corresponding sets of structures. In Proposition~\ref{prop:A3-W4}, we prove that the $\mathcal{W}_4$-component of the intrinsic torsion vanishes for every invariant almost Hermitian structure. Hence, it is enough to consider the classes
\[
\mathcal{W}_I=\bigoplus_{j\in I}\mathcal{W}_j
\]
with $I\subseteq\{1,2,3\},$ since $\mathcal{W}_I$ and $\mathcal{W}_{I\cup\{4\}}$ determine the same set of invariant structures on $\mathbb{F}.$ The sets of structures corresponding to the nonempty classes are smooth submanifolds of the parameter space. Table~\ref{tab:A3-GH-classification} gives the necessary and sufficient conditions for each class and the real dimension, up to homothety, of the corresponding submanifold.\\

The paper is organized as follows. In Section~2, we recall the basic notions of almost Hermitian geometry used throughout the paper and describe the invariant Riemannian metrics and almost complex structures on the maximal real flag manifold of type $\mathcal{A}_3$. In Section~3, we parametrize the invariant almost Hermitian structures and determine their Gray-Hervella classes.

\begin{table}[h]
\centering
\renewcommand{\arraystretch}{1.5}
\begin{tabular}{|c|c|c|}
\hline
\text{Class}
&$\begin{array}{c}
\text{Necessary and}\\
\text{sufficient conditions}
\end{array}$&$\begin{array}{c}
\text{Dimension up to}\\
\text{homothety}
\end{array}$\\[1mm]\hline
$\mathcal{W}_{\emptyset},\mathcal{W}_{2},\mathcal{W}_{3}, \mathcal{W}_{4}$&\text{No invariant structures}&-\\[1mm]\hline
$\mathcal{W}_{1}$&\text{Table~\ref{tab:A3-W1}}&0\\[1mm]\hline
$\mathcal{W}_{\{1,2\}}$&\eqref{eq:A3-W12-triangle}\text{ and }\eqref{eq:A3-W12-z}&2\\[1mm]\hline
$\mathcal{W}_{\{1,3\}}$&\eqref{eq:A3-W13-condition}&4\\[1mm]\hline
$\mathcal{W}_{\{2,3\}}$&\eqref{eq:A3-W23-condition}&6\\[1mm]\hline
$\mathcal{W}_{\{1,2,3\}}$&\eqref{eq:A3-W123}&8\\\hline
\end{tabular}
\caption{Gray-Hervella classes of invariant almost Hermitian structures on $\mathbb{F}$.}
\label{tab:A3-GH-classification}
\end{table}

\section{Preliminaries}
\subsection{Invariant almost Hermitian structures}\label{subsec:almost-Hermitian}
In this subsection, we recall the standard terminology of almost Hermitian geometry and set all sign conventions. We refer to \cite{GH}, \cite[Section~8]{CS} and \cite{N} for the definitions and results reviewed below.\\

Let $G$ be a compact Lie group, let $H\subseteq G$ be a closed subgroup and suppose that the homogeneous space $M:=G/H$ has real dimension $2m.$ Let $o=eH$ be the origin of $M,$ and denote by $\mathfrak{g}$ and $\mathfrak{h}$ the Lie algebras of $G$ and $H,$ respectively. Since $G$ is compact, we may fix an $\operatorname{Ad}(H)$-invariant reductive decomposition
\[
    \mathfrak{g}=\mathfrak{h}\oplus\mathfrak{m}.
\]
If $\pi\colon G\longrightarrow G/H$ is the canonical projection, the restriction of $(d\pi)_e$ to $\mathfrak{m}$ identifies $\mathfrak{m}$ with $T_oM.$ Under this identification, $G$-invariant tensor fields on $M$ correspond to $\operatorname{Ad}(H)$-invariant tensors on $\mathfrak{m}.$ We use the same notation for a $G$-invariant tensor field and its value at $o.$ A $G$-invariant almost complex structure on $M$ is determined by an $\operatorname{Ad}(H)$-equivariant endomorphism $J\colon\mathfrak{m}\longrightarrow\mathfrak{m}$ satisfying $J^2=-\operatorname{Id}_{\mathfrak{m}}.$ Similarly, a $G$-invariant Riemannian metric on $M$ is determined by an $\operatorname{Ad}(H)$-invariant inner product $g$ on $\mathfrak{m}.$ The pair $(g,J)$ is called a $G$-invariant almost Hermitian structure if $g$ and $J$ are compatible, that is,
\begin{equation}\label{eq:compatible-metric}
    g(JX,JY)=g(X,Y),\ X,Y\in\mathfrak{m}.
\end{equation}
Its fundamental form is the $G$-invariant two-form determined at $o$ by
\begin{equation*}
    \omega(X,Y):=g(JX,Y),\ X,Y\in\mathfrak{m}.
\end{equation*}
If $\mathfrak{g}$ is endowed with an $\operatorname{Ad}(G)$-invariant inner product $(\cdot,\cdot),$ then, for every $\operatorname{Ad}(H)$-invariant inner product $g$ on $\mathfrak{m},$ there exists a unique positive-definite, $(\cdot,\cdot)$-self-adjoint and $\operatorname{Ad}(H)$-equivariant linear operator $A\colon\mathfrak{m}\longrightarrow\mathfrak{m}$ such that
\begin{equation}\label{eq:metric-operator}
g(X,Y)=(AX,Y),\ X,Y\in\mathfrak{m}.
\end{equation}
We call $A$ the {\it metric operator} associated with $g.$ Conversely, every positive-definite, $(\cdot,\cdot)$-self-adjoint and $\operatorname{Ad}(H)$-equivariant linear operator on $\mathfrak{m}$ determines an $\operatorname{Ad}(H)$-invariant inner product through \eqref{eq:metric-operator}. Consequently, the $G$-invariant Riemannian metrics on $M$ are in one-to-one correspondence with the linear operators on $\mathfrak{m}$ having these properties. In terms of the metric operator, the compatibility condition \eqref{eq:compatible-metric} is equivalent to
\begin{equation}\label{eq:compatible-metric-operator}
J^{T}AJ=A,
\end{equation}
where $J^{T}$ denotes the adjoint of $J$ with respect to $(\cdot,\cdot).$ The Nijenhuis tensor of $J$ is defined by
\begin{equation*}
    N(X,Y):=-[JX,JY]+J[JX,Y]+J[X,JY]+[X,Y]
\end{equation*}
for $X,Y\in\mathfrak X(M).$ Since $J$ is $G$-invariant, $N$ is also $G$-invariant and its value at $o$ is given by
\begin{equation}\label{eq:homogeneous-Nijenhuis}
    N(X,Y)=-[JX,JY]_{\mathfrak{m}}+J[JX,Y]_{\mathfrak{m}}+J[X,JY]_{\mathfrak{m}}+[X,Y]_{\mathfrak{m}}
\end{equation}
for all $X,Y\in\mathfrak{m}.$ Here $[\cdot,\cdot]_{\mathfrak{m}}$ denotes the $\mathfrak{m}$-component of the Lie bracket of $\mathfrak{g}.$ The almost complex structure $J$ is integrable if $N=0.$ By the Newlander-Nirenberg theorem \cite{NN}, this is equivalent to the existence of a complex atlas on $M$ whose induced almost complex structure is $J.$ In the invariant setting, integrability is therefore equivalent to the vanishing of the right-hand side of \eqref{eq:homogeneous-Nijenhuis} for all $X,Y\in\mathfrak{m}.$ The $G$-invariant almost Hermitian structure $(g,J)$ is called Hermitian if $J$ is integrable, almost K\"ahler if $d\omega=0$ and K\"ahler if $\nabla J=0,$ where $\nabla$ denotes the Levi-Civita connection of $g.$ For any almost Hermitian structure,
\[
    \nabla J=0
    \Longleftrightarrow
    N=0\ \text{and}\ d\omega=0.
\]
Consequently, $(g,J)$ is K\"ahler if and only if it is both Hermitian and almost K\"ahler; see \cite[Theorem~3.1]{GH}. The covariant derivative of $\omega$ is the tensor field defined by
\begin{equation*}
    (\nabla\omega)(X,Y,Z):=(\nabla_X\omega)(Y,Z)=X\bigl(\omega(Y,Z)\bigr)-\omega(\nabla_XY,Z)-\omega(Y,\nabla_XZ)
\end{equation*}
for $X,Y,Z\in\mathfrak X(M).$ The Levi-Civita connection and $\omega$ are $G$-invariant and hence so is $\nabla\omega.$ We may therefore identify $\nabla\omega$ with its value at $o,$ so that
\[
    \nabla\omega\in
    \bigl(\mathfrak m^*\otimes\Lambda^2\mathfrak m^*\bigr)^H.
\]Let $\Lambda^g:\mathfrak{m}\times\mathfrak{m}\longrightarrow\mathfrak{m}$ be the Nomizu map associated with the Levi-Civita connection of $g.$ It is characterized by
\begin{equation*}
    2g(\Lambda^g_XY,Z)=g([X,Y]_{\mathfrak{m}},Z)-g([X,Z]_{\mathfrak{m}},Y)-g(X,[Y,Z]_{\mathfrak{m}})
\end{equation*}
for all $X,Y,Z\in\mathfrak{m}.$ In terms of the Nomizu map, the tensor $\nabla\omega$ is given by
\begin{equation*}
    (\nabla\omega)(X,Y,Z)=-\omega(\Lambda^g_XY,Z)-\omega(Y,\Lambda^g_XZ).
\end{equation*}
In particular, $\nabla\omega$ is completely determined by the Lie bracket of $\mathfrak{g},$ the reductive decomposition and the pair $(g,J).$ Moreover,
\begin{equation}\label{eq:nablaomega-symmetries}
    \begin{aligned}
        (\nabla\omega)(X,Y,Z)=-(\nabla\omega)(X,Z,Y)\ \textnormal{and}\ (\nabla\omega)(X,JY,JZ)=-(\nabla\omega)(X,Y,Z)
    \end{aligned}
\end{equation}
for all $X,Y,Z\in\mathfrak{m}.$ Consequently,
\[
    (\nabla\omega)(X,JY,Z)=(\nabla\omega)(X,Y,JZ).
\]
The exterior derivative of \(\omega\) can be recovered from \(\nabla\omega\) by the formula
\begin{align}\label{eq:domega}
    \begin{split}
    d\omega(X,Y,Z)
    &=(\nabla\omega)(X,Y,Z)+(\nabla\omega)(Y,Z,X)
      +(\nabla\omega)(Z,X,Y)\\
    &=-\omega([X,Y]_{\mathfrak m},Z)-\omega([Y,Z]_{\mathfrak m},X)
      -\omega([Z,X]_{\mathfrak m},Y).
      \end{split}
\end{align}
The codifferential of $\omega$ is given by
\begin{equation*}
    (\delta\omega)(Z)=-\sum_{a=1}^{2m}(\nabla\omega)(e_a,e_a,Z),
\end{equation*}
where $\{e_1,\ldots,e_{2m}\}$ is any $g$-orthonormal basis of $\mathfrak{m}.$ We also regard $N$ as a covariant three-tensor by setting
\[
    N(X,Y,Z):=g(N(X,Y),Z).
\]
In terms of $\nabla\omega$, it is given by
\begin{equation*}
    \begin{split}
        N(X,Y,Z)=&-(\nabla\omega)(JX,Y,Z)
        +(\nabla\omega)(JY,X,Z)\\
        &-(\nabla\omega)(X,Y,JZ)
        +(\nabla\omega)(Y,X,JZ).
    \end{split}
\end{equation*}
Define $N^0:=N-\mathfrak{b}N,$ where
\[
(\mathfrak{b}N)(X,Y,Z):=\frac{1}{3}\left\{N(X,Y,Z)+N(Y,Z,X)+N(Z,X,Y)\right\}
\]
is the skew-symmetric part of $N.$\\

We now recall the Gray-Hervella decomposition \cite{GH}. Assume that $m\geq3$ and define
\[
    \mathcal K(\mathfrak{m}):=\left\{A\in\mathfrak{m}^*\otimes\Lambda^2\mathfrak{m}^*:A(X,JY,JZ)=-A(X,Y,Z)
    \right\}.
\]
By \eqref{eq:nablaomega-symmetries}, $\nabla\omega\in\mathcal K(\mathfrak{m})^H.$ The unitary group of the Hermitian vector space $(\mathfrak{m},g,J)$ is 
\[
\mathrm{U}(\mathfrak{m},g,J):=\{u\in\mathrm{GL}(\mathfrak{m}):g(uX,uY)=g(X,Y),\ uJ=Ju\}.
\]
The Gray-Hervella decomposition is the orthogonal direct sum
\begin{equation*}
    \mathcal{K}(\mathfrak{m})=\mathcal{W}_1\oplus\mathcal{W}_2\oplus\mathcal{W}_3\oplus\mathcal{W}_4,
\end{equation*}
where the summands are irreducible  and inequivalent real $\mathrm U(\mathfrak{m},g,J)$-modules. Each summand will be characterized below in terms of tensors derived from $d\omega$ and the covariant Nijenhuis tensor. For $I\subseteq\{1,2,3,4\}$, set
\[
    \mathcal{W}_I:=\bigoplus_{i\in I}\mathcal{W}_i,\ \mathcal W_\emptyset:=\{0\}.
\]
The $G$-invariant almost Hermitian structure $(g,J)$ belongs to the Gray-Hervella class $\mathcal{W}_I$ if $\nabla\omega\in\mathcal{W}_I.$ Equivalently, the $\mathcal{W}_j$-component of $\nabla\omega$ vanishes for every $j\notin I.$ The Gray-Hervella class is unchanged under homotheties of the metric and under a change of sign of $J,$ that is, for every $c>0,$
\[
(g,J)\in\mathcal{W}_I
\Longleftrightarrow(cg,J)\in\mathcal{W}_I
\Longleftrightarrow(g,-J)\in\mathcal{W}_I.
\]
Set $\mathfrak{m}_{\mathbb{C}}:=\mathfrak{m}\otimes_{\mathbb{R}}\mathbb{C}$ and extend $J$, $g$, $\omega$, $\nabla\omega$, $d\omega$ and $N$ by complex linearity. Then $J$ determines the decomposition
\[
    \mathfrak{m}_{\mathbb C}=\mathfrak{m}^{1,0}\oplus\mathfrak{m}^{0,1},
\]
where $\mathfrak{m}^{1,0}:=\left\{
        Z\in\mathfrak{m}_{\mathbb{C}}:JZ=iZ\right\}$ and $\mathfrak{m}^{0,1}:=\left\{
        Z\in\mathfrak{m}_{\mathbb{C}}:JZ=-iZ\right\}.$ For $X\in\mathfrak{m}$, the corresponding projections are
\[
    X^{1,0}=\frac{1}{2}(X-iJX),\ X^{0,1}=\frac12(X+iJX).
\]
In particular, $\mathfrak{m}^{0,1}=\overline{\mathfrak{m}^{1,0}}.$ If $J^*\colon\mathfrak{m}_{\mathbb{C}}^*\longrightarrow\mathfrak{m}_{\mathbb{C}}^*$ is the endomorphism defined by $(J^*\alpha)(X):=\alpha(JX),$ then
\[    \mathfrak{m}_{\mathbb{C}}^*=\Lambda^{1,0}\mathfrak{m}^*\oplus\Lambda^{0,1}\mathfrak{m}^*,
\]
where
\begin{align*}
    &\Lambda^{1,0}\mathfrak{m}^*:=\left\{\alpha\in\mathfrak{m}_{\mathbb{C}}^*:J^*\alpha=i\alpha\right\}\\ &\Lambda^{0,1}\mathfrak{m}^*:=\left\{\alpha\in\mathfrak{m}_{\mathbb C}^*:J^*\alpha=-i\alpha\right\}=\overline{\Lambda^{1,0}\mathfrak{m}^*}.
\end{align*}
For $p,q\geq0,$ define
\[
\Lambda^{p,q}\mathfrak{m}^*:=\Lambda^p\left(\Lambda^{1,0}\mathfrak{m}^*\right)\wedge\Lambda^q\left(\Lambda^{0,1}\mathfrak{m}^*\right).
\]
It follows that
\[
\Lambda^k\mathfrak{m}_{\mathbb{C}}^*=\bigoplus_{p+q=k}\Lambda^{p,q}\mathfrak{m}^*.
\]
In particular, since $d\omega\in\Lambda^3\mathfrak{m}_{\mathbb{C}}^*,$ we can write
\[
d\omega=(d\omega)^{3,0}+(d\omega)^{0,3}+(d\omega)^{2,1}+(d\omega)^{1,2},
\]
where $(d\omega)^{p,q}\in\Lambda^{p,q}\mathfrak{m}^*.$ We set $(d\omega)^-:=(d\omega)^{3,0}+(d\omega)^{0,3}$ and $(d\omega)^+:=(d\omega)^{2,1}+(d\omega)^{1,2},$ so that $d\omega=(d\omega)^-+(d\omega)^+.$ A $3$-form $\eta$ is called primitive if $\eta\wedge\omega^{m-2}=0.$ The component $(d\omega)^-$ is always primitive, while $(d\omega)^+$ admits the decomposition
\[
(d\omega)^+=(d\omega)^+_0+\frac{1}{m-1}\alpha_\omega\wedge\omega,
\]
where $\alpha_\omega:=-J^*\delta\omega$ is the Lee form of $\omega$ and $(d\omega)^+_0$ is primitive. Thus, the Lee form determines the non-primitive component of $d\omega.$ In particular, $d\omega$ is primitive if and only if $\alpha_\omega=0.$\\

The following characterization follows from
\cite[Proposition~1]{G} and the Gray-Hervella decomposition.
\begin{proposition}
    Let $(g,J)$ be a $G$-invariant almost Hermitian structure on $M=G/H,$ where $\dim_\mathbb{R}M=2m$ and $m\geq 3.$ Write 
    \[
    \nabla\omega=(\nabla\omega)_1+(\nabla\omega)_2+(\nabla\omega)_3+(\nabla\omega)_4,
    \]
    where $(\nabla\omega)_i\in\mathcal{W}_i,\ i=1,2,3,4.$ Then 
    \begin{itemize}
        \item $(\nabla\omega)_1=0$ if and only if $(d\omega)^-=0,$
        \item $(\nabla\omega)_2=0$ if and only if $N^0=0,$
        \item $(\nabla\omega)_3=0$ if and only if $(d\omega)^+_0=0,$
        \item $(\nabla\omega)_4=0$ if and only if $\alpha_\omega=0.$
    \end{itemize}
\end{proposition}
It follows immediately that $(g,J)$ belongs to the Gray-Hervella class $\mathcal{W}_I,\ (I\subseteq\{1,2,3,4\})$ if and only if
\begin{itemize}
    \item $(d\omega)^-=0$ whenever $1\notin I,$
    \item $N^0=0$ whenever $2\notin I,$
    \item $(d\omega)^+_0=0$ whenever $3\notin I,$
    \item $\alpha_\omega=0$ whenever $4\notin I.$
\end{itemize}
We recall the usual terminology for some of the Gray-Hervella classes. The almost Hermitian structure $(g,J)$ is called
\begin{itemize}
    \item {\it K\"ahler} if it belongs to $\mathcal{W}_{\emptyset};$
    \item {\it nearly K\"ahler} if it belongs to $\mathcal{W}_1;$
    \item {\it almost K\"ahler} if it belongs to $\mathcal{W}_2;$
    \item {\it Hermitian semi-K\"ahler}, also called {\it balanced Hermitian}, if it belongs to $\mathcal{W}_3;$
    \item {\it locally conformally K\"ahler} if it belongs to $\mathcal{W}_4;$
    \item {\it quasi-K\"ahler} if it belongs to $\mathcal{W}_1\oplus\mathcal{W}_2$;
    \item {\it Hermitian} if it belongs to $\mathcal{W}_3\oplus\mathcal{W}_4;$
    \item {\it semi-K\"ahler} or {\it co-symplectic} if it belongs to $\mathcal{W}_1\oplus\mathcal{W}_2\oplus\mathcal{W}_3.$
\end{itemize}
We conclude this section with the following observation, which will be useful later.

\begin{remark}\label{rem:type-reductions}
Since $d\omega$ is a real $3$-form, its components $(d\omega)^{3,0}$ and $(d\omega)^{0,3}$ are complex conjugate. Hence,
\[
(d\omega)^-=0
\ \textnormal{if and only if}\
\left.d\omega\right|_{\mathfrak{m}^{1,0}\times\mathfrak{m}^{1,0}\times\mathfrak{m}^{1,0}}=0.
\]
Similarly, since $(d\omega)^+_0$ is a real $3$-form of type $(2,1)+(1,2),$
\[
(d\omega)^+_0=0
\ \textnormal{if and only if}\
\left.(d\omega)^+_0\right|_{\mathfrak{m}^{0,1}\times\mathfrak{m}^{1,0}\times\mathfrak{m}^{1,0}}=0.
\]
On the other hand, the identities
\[
N(JX,Y,Z)=N(X,JY,Z)=N(X,Y,JZ)
\]
imply that $N$ vanishes whenever its arguments are not all of the same type. The same conclusion holds for $\mathfrak{b}N$ and hence for $N^0.$ Since these are real tensors, their restrictions to $\mathfrak{m}^{1,0}\times \mathfrak{m}^{1,0}\times \mathfrak{m}^{1,0}$ and $\mathfrak{m}^{0,1}\times\mathfrak{m}^{0,1}\times \mathfrak{m}^{0,1}$ are complex conjugate. Therefore, 
\[
N^0=0
\ \textnormal{if and only if}\ 
\left.N^0\right|_{\mathfrak{m}^{1,0}\times
\mathfrak{m}^{1,0}\times\mathfrak{m}^{1,0}}=0.
\]
\end{remark}

\subsection{The maximal real flag manifold of type $\mathcal{A}_3$}
\label{sec:A3-real-flag}

In this subsection, we describe the invariant Riemannian metrics and the invariant almost complex structures of the maximal real flag manifold of type $\mathcal{A}_3.$\\

Let $\mathrm{SL}(4,\mathbb{R})$ be the special linear group of degree $4$ over $\mathbb{R}$ and let $\mathfrak{sl}(4,\mathbb{R})$ be its Lie algebra. Consider the Cartan involution $\tau:\mathfrak{sl}(4,\mathbb{R})\to\mathfrak{sl}(4,\mathbb{R})$ defined by $\tau(X):=-X^{T}.$ The corresponding Cartan decomposition is $\mathfrak{sl}(4,\mathbb{R})=\mathfrak{so}(4)\oplus\mathfrak{s},$ where
\[
\mathfrak{s}=\{X\in\mathfrak{sl}(4,\mathbb{R}):X^{T}=X\}.
\]
A maximal abelian subalgebra of $\mathfrak{s}$ is $\mathfrak{a}=\left\{\operatorname{diag}(a_1,a_2,a_3,a_4):a_1+a_2+a_3+a_4=0\right\}.$ For $1\leq i\neq j\leq 4,$ define $\alpha_{ij}:\mathfrak{a}\to\mathbb{R}$ by
\[
\alpha_{ij}\left(\operatorname{diag}(a_1,a_2,a_3,a_4)\right):=a_i-a_j.
\]
Then the root system of $\mathfrak{sl}(4,\mathbb{R})$ associated with $\mathfrak{a}$ is $\Pi=\{\alpha_{ij}:1\leq i\neq j\leq 4\}$ and the corresponding root spaces are
\[
\mathfrak{g}_{\alpha_{ij}}=\mathbb{R}E_{ij},\ 1\leq i\neq j\leq 4,
\]
where $E_{ij}$ is the $4\times 4$ matrix with a 1 in its $(i,j)$-entry and zero elsewhere. Choosing $\Pi^+=\{\alpha_{ij}:1\leq i<j\leq4\},$ the associated minimal parabolic subalgebra is the subalgebra $\mathfrak{p}$ of upper triangular matrices in $\mathfrak{sl}(4,\mathbb{R}).$ Let $P$ be the normalizer of $\mathfrak{p}$ in $\operatorname{SL}(4,\mathbb{R}).$ The maximal real flag manifold of type $\mathcal{A}_3$ is 
\[
\mathbb{F}:=\mathrm{SL}(4,\mathbb{R})/P.
\]

The group $\mathrm{SO}(4)$ acts transitively on $\mathbb{F}$ with isotropy group
\[
\begin{aligned}
\mathrm{SO}(4)\cap P&=\left\{\operatorname{diag}(\varepsilon_1,\varepsilon_2,\varepsilon_3,\varepsilon_4):\varepsilon_j\in\{1,-1\},\ \varepsilon_1\varepsilon_2\varepsilon_3\varepsilon_4=1\right\}\\
&=\mathrm{S}\left(\mathrm{O}(1)\times\mathrm{O}(1)\times\mathrm{O}(1)\times\mathrm{O}(1)\right).
\end{aligned}
\]
Hence, 
\[
\mathbb{F}=\mathrm{SO}(4)/\mathrm{S}(\mathrm{O}(1)\times\mathrm{O}(1)\times\mathrm{O}(1)\times\mathrm{O}(1)).
\]
Let $B$ denote the Killing form of $\mathfrak{sl}(4,\mathbb{R})$ and endow $\mathfrak{so}(4)$ with the $\operatorname{Ad}(\operatorname{SO}(4))$-invariant inner product
\begin{equation}\label{eq:A3-inner-product}
(X,Y):=-\frac{1}{4}B(X,Y)
=-2\operatorname{Tr}(XY),\ X,Y\in\mathfrak{so}(4).
\end{equation}
The matrices 
\[
\begin{aligned}
X_{1}&=\frac{1}{2}(E_{21}-E_{12}),&
X_{2}&=\frac{1}{2}(E_{43}-E_{34}),\\
X_{3}&=\frac{1}{2}(E_{31}-E_{13}),&
X_{4}&=\frac{1}{2}(E_{42}-E_{24}),\\
X_{5}&=\frac{1}{2}(E_{32}-E_{23}),&
X_{6}&=\frac{1}{2}(E_{41}-E_{14})
\end{aligned}
\]
form an ordered orthonormal basis of
$\mathfrak{so}(4)$ with respect to $(\cdot,\cdot).$ We denote this basis by
\begin{equation}\label{eq:A3-basis}
\mathcal{B}:=(X_1,X_2,X_3,X_4,X_5,X_6).
\end{equation}
Since the isotropy group is discrete, the $\mathrm{SO}(4)$-invariant Riemannian metrics are determined by $\mathrm{Ad}(\mathrm{S}(\mathrm{O}(1)\times\mathrm{O}(1)\times\mathrm{O}(1)\times\mathrm{O}(1)))$-invariant inner products  on $\mathfrak{so}(4)$ and the $\mathrm{SO}(4)$-invariant almost complex structures on $\mathbb{F}$ are determined by $\mathrm{Ad}(\mathrm{S}(\mathrm{O}(1)\times\mathrm{O}(1)\times\mathrm{O}(1)\times\mathrm{O}(1)))$-equivariant endomorphisms $J:\mathfrak{so}(4)\to\mathfrak{so}(4)$ satisfying $J^2=-\mathrm{Id}_{\mathfrak{so}(4)}.$

\begin{proposition}[\GGNTheoremA]
A linear map $A:\mathfrak{so}(4)\to\mathfrak{so}(4)$ is the metric operator associated with an $\operatorname{SO}(4)$-invariant Riemannian metric on $\mathbb{F}$ if and only if its matrix with respect to $\mathcal{B}$ has the form
\begin{equation}\label{eq:A3-metric-operator}
[A]_{\mathcal{B}}=\operatorname{diag}(A_{1},A_{2},A_{3}),
\end{equation}
where
\begin{equation*}\label{eq:A3-metric-blocks}
A_{j}=\begin{pmatrix}
\mu_{j} & b_{j}\\
b_{j} & \xi_{j}
\end{pmatrix},
\ \mu_{j},\xi_{j}>0,\ \Delta_{j}:=\mu_{j}\xi_{j}-b_{j}^{2}>0, j=1,2,3.
\end{equation*}
\end{proposition}
\begin{proposition}[\FBSMA]\label{prop:almost-complex}
A linear map $J\colon\mathfrak{so}(4)\to\mathfrak{so}(4)$ is an $\operatorname{SO}(4)$-invariant almost complex structure on $\mathbb{F}$ if and only if its matrix with respect to $\mathcal{B}$ has the form
\begin{equation}\label{eq:A3-almost-complex-structure}
[J]_{\mathcal{B}}=\operatorname{diag}(J_{1},J_{2},J_{3}),
\end{equation}
where
\[
J_{j}=\begin{pmatrix}
\alpha_{j} & \beta_{j}\\
\gamma_{j} & -\alpha_{j}
\end{pmatrix},
\ \alpha_{j}^{2}+\beta_{j}\gamma_{j}=-1,
\ j=1,2,3.
\]
None of these structures are integrable.
\end{proposition}

\section{The Gray-Hervella classes}

In this section, we describe the $\mathrm{SO}(4)$-invariant almost Hermitian structures on the maximal real flag manifold of type $\mathcal{A}_3$ and characterize the corresponding Gray-Hervella classes.\\

Consider the inner product $(\cdot,\cdot)$ defined in \eqref{eq:A3-inner-product}. Let $g$ be an $\operatorname{SO}(4)$-invariant Riemannian metric on $\mathbb{F},$ with metric operator $A$ of the form \eqref{eq:A3-metric-operator} and let $J$ be an invariant almost complex structure of the form \eqref{eq:A3-almost-complex-structure}.

\begin{proposition}\label{prop:A3-almost-Hermitian}
The pair $(g,J)$ is an invariant almost Hermitian structure if and only if
there exist $\epsilon_{1},\epsilon_{2},\epsilon_{3}\in\{1,-1\}$ such that
\begin{equation}\label{eq:A3-almost-Hermitian}
\alpha_j=\epsilon_j\frac{b_j}{\sqrt{\Delta_j}},\ 
\beta_j=\epsilon_j\frac{\xi_j}{\sqrt{\Delta_j}},\ 
\gamma_j=-\epsilon_j\frac{\mu_j}{\sqrt{\Delta_j}},\ 
j=1,2,3.
\end{equation}
\end{proposition}

\begin{proof}
By \eqref{eq:compatible-metric-operator}, $g$ and $J$ are compatible if and only if $J_j^T A_jJ_j=A_j$ for $j=1,2,3.$ This condition is equivalent to 
\begin{equation*}
    \begin{cases}
        \mu_j\alpha_j+b_j\gamma_j=0,\\
        \mu_j\beta_j+\xi_j\gamma_j=0,\\
        b_j\beta_j-\xi_j\alpha_j=0.
    \end{cases}
\end{equation*}Solving this linear system yields $(\alpha_j,\beta_j,\gamma_j)=\lambda_j(b_j,\xi_j,-\mu_j)$ for some $\lambda_j\in\mathbb{R}.$ Since $\alpha_j^2+\beta_j\gamma_j=-1,$ there exist $\epsilon_j\in\{1,-1\}$ such that $\lambda_j=\epsilon_j/\sqrt{\Delta_j}.$ Therefore, \eqref{eq:A3-almost-Hermitian} holds. Conversely, if $J$ is defined by $\alpha_j,\beta_j,\gamma_j$ as in \eqref{eq:A3-almost-Hermitian}, then it is straightforward to verify that $\alpha_j^2+\beta_j\gamma_j=-1$ and $J_j^TA_jJ_j=A_j,$ for $j=1,2,3.$ Thus, $(g,J)$ is an invariant almost Hermitian structure.
\end{proof}

In the preceding proposition, the signs $\epsilon_1,\epsilon_2,\epsilon_3\in\{1,-1\}$ may be chosen arbitrarily. Consequently, the set of all invariant almost Hermitian structures on $\mathbb{F}$ is parametrized by the set
\begin{equation*}\label{eq:A3-parameters-almost-Hermitian}
\mathcal{P}:=\left\{(\mu_j,\xi_j,b_j,\epsilon_j)_{j=1}^{3}\in\mathbb{R}^9\times\{1,-1\}^3:\mu_j,\xi_j,\Delta_j>0,\ j=1,2,3\right\}.
\end{equation*}

We can also obtain another parametrization as follows. Let
\begin{equation*}\label{eq:A3-complex-parameters}
\mathcal{Z}:=\left\{(z_j,\sigma_j)_{j=1}^{3}\in\mathbb{C}^3\times\mathbb{R}^3:\sigma_j\operatorname{Im}(z_j)>0,\ j=1,2,3\right\},
\end{equation*}
and consider the map
\[
\mathcal{P}\ni(\mu_j,\xi_j,b_j,\epsilon_j)_{j=1}^3\longmapsto(z_j,\sigma_j)_{j=1}^3\in\mathcal{Z}
\]
defined by
\begin{equation}\label{eq:P-to-Z}
z_j=\frac{b_j+i\sigma_j}{\mu_j},\ \sigma_j=\epsilon_j\sqrt{\Delta_j},\ j=1,2,3.
\end{equation}
Then, it is easy to verify that this map is a diffeomorphism whose inverse is given by
\begin{equation}\label{eq:Z-to-P}
\mu_j=\frac{\sigma_j}{\operatorname{Im}(z_j)},\ b_j=\frac{\sigma_j\operatorname{Re}(z_j)}{\operatorname{Im}(z_j)},\ \xi_j=\frac{\sigma_j|z_j|^2}{\operatorname{Im}(z_j)},\ \epsilon_j=\frac{\sigma_j}{|\sigma_j|},\ j=1,2,3.
\end{equation}

For the rest of this section, we assume that $(g,J)$ is an invariant
almost Hermitian structure. Let
$(\theta^1,\theta^2,\theta^3,\theta^4,\theta^5,\theta^6)$ denote the dual basis of $\mathcal{B},$ as defined in \eqref{eq:A3-basis}, and write $\theta^{jk}:=\theta^j\wedge\theta^k,$ $\theta^{jk\ell}:=\theta^j\wedge\theta^k\wedge\theta^\ell.$ The fundamental form $\omega$ associated with $(g,J)$ is given by
\[
\omega(X,Y)=g(JX,Y)=(AJX,Y),
\ X,Y\in\mathfrak{so}(4).
\]
Equivalently,
\[
\omega(X,Y)=[X]_{\mathcal{B}}^{T}
[J]_{\mathcal{B}}^{T}[A]_{\mathcal{B}}[Y]_{\mathcal{B}},
\]
where $[X]_{\mathcal{B}}$ and $[Y]_{\mathcal{B}}$ are the coordinate vectors of $X$ and $Y$ with respect to $\mathcal{B}$, regarded as column vectors. Since $\mathcal{B}$ is $(\cdot,\cdot)$-orthonormal, we have that $\omega(X_r,X_s)$ is the $(r,s)$-entry of the matrix $[J]_{\mathcal{B}}^{T}[A]_{\mathcal{B}}.$ By \eqref{eq:A3-almost-Hermitian} and \eqref{eq:P-to-Z},
\begin{equation*}
[J]_{\mathcal{B}}^T
[A]_{\mathcal{B}}
=
\mathrm{diag}\left(
\sigma_1\begin{pmatrix}0&-1\\1&0\end{pmatrix},
\sigma_2\begin{pmatrix}0&-1\\1&0\end{pmatrix},
\sigma_3\begin{pmatrix}0&-1\\1&0\end{pmatrix}
\right).
\end{equation*}
Therefore,
\begin{equation*}
\omega
=
-\sigma_1\theta^{12}
-\sigma_2\theta^{34}
-\sigma_3\theta^{56}.
\end{equation*}

The nonzero Lie brackets among $X_1,...,X_6$ are
\begin{equation*}
\begin{array}{lll}
\displaystyle [X_1,X_3]=\frac{1}{2}X_5,
&
\displaystyle [X_1,X_4]=-\frac{1}{2}X_6,
&
\displaystyle [X_1,X_5]=-\frac{1}{2}X_3,
\\[1em]
\displaystyle [X_1,X_6]=\frac{1}{2}X_4,
&
\displaystyle [X_2,X_3]=\frac{1}{2}X_6,
&
\displaystyle [X_2,X_4]=-\frac{1}{2}X_5,
\\[1em]
\displaystyle [X_2,X_5]=\frac{1}{2}X_4,
&
\displaystyle [X_2,X_6]=-\frac{1}{2}X_3,
&
\displaystyle [X_3,X_5]=\frac{1}{2}X_1,
\\[1em]
\displaystyle [X_3,X_6]=\frac{1}{2}X_2,
&
\displaystyle [X_4,X_5]=-\frac{1}{2}X_2,
&
\displaystyle [X_4,X_6]=-\frac{1}{2}X_1.
\end{array}
\end{equation*}

Using \eqref{eq:domega} and the bracket relations above, we have that the only possibly nonzero values of $d\omega$ (up to permutations of its entries) are
\begin{equation*}
\begin{aligned}
d\omega(X_1,X_3,X_6)
&=\frac12(-\sigma_1+\sigma_2+\sigma_3),\\
d\omega(X_1,X_4,X_5)
&=\frac12(\sigma_1+\sigma_2+\sigma_3),\\
d\omega(X_2,X_3,X_5)
&=\frac12(\sigma_1+\sigma_2-\sigma_3),\\
d\omega(X_2,X_4,X_6)
&=\frac12(-\sigma_1+\sigma_2-\sigma_3).
\end{aligned}
\end{equation*}
Therefore,
\begin{equation}\label{eq:A3-domega}
d\omega
=
\frac12
\left(
A\theta^{136}
+B\theta^{145}
+C\theta^{235}
+D\theta^{246}
\right),
\end{equation}
where
\begin{equation*}
\begin{aligned}
&A=-\sigma_1+\sigma_2+\sigma_3,
&
&B=\sigma_1+\sigma_2+\sigma_3,\\
&C=\sigma_1+\sigma_2-\sigma_3,
&
&D=-\sigma_1+\sigma_2-\sigma_3.
\end{aligned}
\end{equation*}

%\begin{proposition}
%There are no invariant almost Hermitian structures on $\mathbb{F}$ belonging to the Gray-Hervella class $\mathcal{W}_{\{3,4\}}$.
%\end{proposition}
%\begin{proof}
   %As the class $\mathcal{W}_{\{3,4\}}$ corresponds to integrable structures, the proof follows directly from the proposition \ref{prop:almost-complex}.
%\end{proof}

\begin{proposition}\label{prop:A3-W4}
For every invariant almost Hermitian structure on $\mathbb{F},$ the $\mathcal{W}_4$-component of $\nabla\omega$ vanishes, i.e. every pair $(g,J)$ is co-symplectic.
\end{proposition}

\begin{proof}
Observe that $d\omega$ is a linear combination of $\theta^{136},\theta^{145},\theta^{235},\theta^{246},$ and each of these $3$-forms contains exactly one index from each of the pairs $\{1,2\},\{3,4\},\{5,6\}.$ Therefore, its wedge product with each of $\theta^{12},\theta^{34},\theta^{56}$ vanishes, since at least one index is repeated. Hence, $d\omega\wedge\omega=0.$ Since $\dim\mathbb{F}=6,$ $d\omega$ is primitive. Thus, $\alpha_\omega=0$ and the $\mathcal{W}_4$-component of $\nabla\omega$ vanishes.
\end{proof}

\begin{corollary}
There are the following equivalences:

I) $\mathcal{W}_{\{1,4\}} \approx \mathcal{W}_1$;
II) $\mathcal{W}_{\{2,4\}} \approx \mathcal{W}_2$;
III) $\mathcal{W}_{\{3,4\}} \approx \mathcal{W}_3$

IV) $\mathcal{W}_{\{1,2,4\}} \approx \mathcal{W}_{1,2}$;
V) $\mathcal{W}_{\{1,3,4\}} \approx \mathcal{W}_{\{1,3\}}$;
VI) $\mathcal{W}_{\{2,3,4\}} \approx \mathcal{W}_{\{2,3\}}$

VII) $\mathcal{W}_{\{1,2,3\}} \approx \mathcal{W}_{\{1,2,3,4\}}$
\end{corollary}

To obtain information about the other components of $\nabla\omega,$ set
\[
Z_j:=X_{2j}-z_jX_{2j-1},\ j=1,2,3,
\]
where the $z_j$ are given in \eqref{eq:P-to-Z}. Then $\{Z_1,Z_2,Z_3\}$ is a basis of $\mathfrak{m}^{1,0}$ and equation \eqref{eq:A3-domega} gives
\begin{align*}
d\omega(Z_1,Z_2,Z_3)
&=
\frac12
\left(
Az_1z_2+Bz_1z_3+Cz_2z_3+D
\right),
\\[1em]
d\omega(\overline{Z_1},Z_2,Z_3)
&=
\frac12
\left(
A\overline{z_1}z_2
+B\overline{z_1}z_3
+Cz_2z_3+D
\right),
\\[1em]
d\omega(Z_1,\overline{Z_2},Z_3)
&=
\frac12
\left(
Az_1\overline{z_2}
+Bz_1z_3
+C\overline{z_2}z_3+D
\right),
\\[1em]
d\omega(Z_1,Z_2,\overline{Z_3})
&=
\frac12
\left(
Az_1z_2
+Bz_1\overline{z_3}
+Cz_2\overline{z_3}+D
\right).
\end{align*}

Using \eqref{eq:homogeneous-Nijenhuis}, we also obtain
\[
\begin{aligned}
&N(Z_2,Z_3,Z_1)=-2i\sigma_1\left(z_2z_3-z_1z_2+z_1z_3-1\right),\\
&N(Z_3,Z_1,Z_2)=-2i\sigma_2\left(z_1z_2+z_1z_3+z_2z_3+1\right),\\
&N(Z_1,Z_2,Z_3)=-2i\sigma_3\left(z_1z_2+z_1z_3-z_2z_3-1\right).
\end{aligned}
\]
The only possibly nonzero values of $N^0\big{|}_{\mathfrak{m}^{1,0}\times\mathfrak{m}^{1,0}\times\mathfrak{m}^{1,0}}$ are
\[
\begin{aligned}
N^0(Z_2,Z_3,Z_1)&=N(Z_2,Z_3,Z_1)-\mathfrak{b}N(Z_1,Z_2,Z_3),\\
N^0(Z_3,Z_1,Z_2)
&=N(Z_3,Z_1,Z_2)-\mathfrak{b}N(Z_1,Z_2,Z_3),\\
N^0(Z_1,Z_2,Z_3)&=N(Z_1,Z_2,Z_3)-\mathfrak{b}N(Z_1,Z_2,Z_3).
\end{aligned}
\]
Consequently, $N^0\big|_{\mathfrak{m}^{1,0}\times\mathfrak{m}^{1,0}\times\mathfrak{m}^{1,0}}=0$ if and only if
\[
N(Z_2,Z_3,Z_1)=N(Z_3,Z_1,Z_2)=N(Z_1,Z_2,Z_3).
\]

By Remark~\ref{rem:type-reductions} and the computations above, we have
\begin{align}
(\nabla\omega)_1=0
&\Longleftrightarrow
Az_1z_2+Bz_1z_3+Cz_2z_3+D=0,
\label{eq:A3-W1-condition}
\\
(\nabla\omega)_2=0
&\Longleftrightarrow
\begin{cases}
\sigma_1(z_2z_3-z_1z_2+z_1z_3-1)
=
\sigma_2(z_1z_2+z_1z_3+z_2z_3+1),
\\
\sigma_2(z_1z_2+z_1z_3+z_2z_3+1)
=
\sigma_3(z_1z_2+z_1z_3-z_2z_3-1),
\end{cases}
\label{eq:A3-W2-condition}
\\
(\nabla\omega)_3=0
&\Longleftrightarrow
\begin{cases}
A\overline{z_1}z_2
+B\overline{z_1}z_3
+Cz_2z_3+D=0,
\\
Az_1\overline{z_2}
+Bz_1z_3
+C\overline{z_2}z_3+D=0,
\\
Az_1z_2
+Bz_1\overline{z_3}
+Cz_2\overline{z_3}+D=0,
\end{cases}
\label{eq:A3-W3-condition}
\end{align}
where each system must be considered subject to 
\begin{equation}\label{eq:A3-admissibility}
\sigma_j\operatorname{Im}(z_j)>0,\ j=1,2,3.
\end{equation}
\begin{proposition}\label{prop:A3-empty-classes}
For every $I\in\{\emptyset,\{2\},\{3\},\{4\}\},$ there are no invariant almost Hermitian structures on $\mathbb{F}$ belonging to the Gray-Hervella class $\mathcal{W}_I.$
\end{proposition}
\begin{proof}
Suppose first that $(g,J)$ belongs to $\mathcal{W}_2.$ Then $(d\omega)^-=(d\omega)^+_0=0.$ Moreover, by Proposition~\ref{prop:A3-W4}, $\alpha_\omega=0,$ so $d\omega=0.$ By \eqref{eq:A3-domega}, this implies that $A=B=C=D=0.$ However, $B-A=2\sigma_1\neq0,$ which is a contradiction. Hence, there are no invariant almost Hermitian structures belonging to $\mathcal{W}_2.$ Since $\mathcal{W}_{\emptyset}\subseteq\mathcal{W}_2,$ the same conclusion holds for $\mathcal{W}_{\emptyset}.$ Finally, if $(g,J)$ belongs to $\mathcal{W}_3$ or to $\mathcal{W}_4$ , then $J$ is integrable, which contradicts Proposition~\ref{prop:almost-complex}. Therefore, there are no invariant almost Hermitian structures belonging to $\mathcal{W}_3$ or to $\mathcal{W}_4$.
\end{proof}

\begin{proposition}\label{prop:A3-W12}
An invariant almost Hermitian structure $(z_j,\sigma_j)_{j=1}^{3}$ on $\mathbb{F}$ belongs to the Gray-Hervella class $\mathcal{W}_{\{1,2\}}$ if and only if
\begin{equation}\label{eq:A3-W12-triangle}
|\sigma_1|<|\sigma_2|+|\sigma_3|,\ |\sigma_2|<|\sigma_1|+|\sigma_3|,\ |\sigma_3|<|\sigma_1|+|\sigma_2|,
\end{equation}
and
\begin{equation}\label{eq:A3-W12-z}
z_1=i\frac{\sigma_1}{|\sigma_1|}\sqrt{-\frac{CD}{AB}},\ z_2=i\frac{\sigma_2}{|\sigma_2|}\sqrt{-\frac{BD}{AC}},\ z_3=i\frac{\sigma_3}{|\sigma_3|}\sqrt{-\frac{AD}{BC}}.
\end{equation}
Consequently, up to homothety, the set of invariant almost Hermitian structures on $\mathbb{F}$ that belong to $\mathcal{W}_{\{1,2\}}$ is a disjoint union of eight manifolds, each of real dimension $2.$
\end{proposition}

\begin{proof}
By Proposition~\ref{prop:A3-W4}, the $\mathcal{W}_4$-component of $\nabla\omega$ vanishes. Hence, $(z_j,\sigma_j)_{j=1}^3\in\mathcal{Z}$ belongs to $\mathcal{W}_{\{1,2\}}$ if and only if $(\nabla\omega)_3=0.$ By \eqref{eq:A3-W3-condition}, this is equivalent to
\begin{equation}\label{eq:auxiliary-A3-W12}
f_1=f_2=f_3=0,
\end{equation}
where
\begin{align*}
&f_1:=A\overline{z_1}z_2+B\overline{z_1}z_3+Cz_2z_3+D,\\
&f_2:=Az_1\overline{z_2}+Bz_1z_3+C\overline{z_2}z_3+D,\\
&f_3:=Az_1z_2+Bz_1\overline{z_3}+Cz_2\overline{z_3}+D.
\end{align*}
Assume that \eqref{eq:auxiliary-A3-W12} holds. Since $A,B,C,D\in\mathbb{R},$ we obtain
\begin{align*}
&0=\overline{f_1}-f_2=(\overline{z_3}-z_3)(Bz_1+C\overline{z_2}),\\
&0=\overline{f_1}-f_3=(\overline{z_2}-z_2)(Az_1+C\overline{z_3}),\\
&0=\overline{f_2}-f_3=(\overline{z_1}-z_1)(Az_2+B\overline{z_3}).
\end{align*}
Since $\sigma_j\operatorname{Im}(z_j)>0,$ we have $\overline{z_j}-z_j\neq0,$ $j=1,2,3.$ Therefore,
\begin{equation}\label{eq:A3-W12-linear-relations}
Bz_1+C\overline{z_2}=0,\ Az_1+C\overline{z_3}=0,\ Az_2+B\overline{z_3}=0.
\end{equation}
If $ABC=0,$ then the equations above and the fact that $z_j\neq0,$ $j=1,2,3,$ imply that $A=B=C=0,$ contradicting $B-A=2\sigma_1\neq0.$ Thus, $ABC\neq0.$ Writing $z_j=x_j+iy_j$ and taking real parts in \eqref{eq:A3-W12-linear-relations} yields
\[
\begin{pmatrix}
B&C&0\\
A&0&C\\
0&A&B
\end{pmatrix}
\begin{pmatrix}
x_1\\x_2\\x_3
\end{pmatrix}
=0.
\]
The determinant of the matrix above is $-2ABC\neq0$ and hence $x_1=x_2=x_3=0.$ Thus, $z_j=iy_j,$ $y_j\neq0$ and $\sigma_jy_j>0,$ $j=1,2,3.$ Substituting $z_j=iy_j$ into \eqref{eq:A3-W3-condition} gives
\begin{equation*}
\begin{cases}
Ay_1y_2+By_1y_3-Cy_2y_3+D=0,\\
Ay_1y_2-By_1y_3+Cy_2y_3+D=0,\\
-Ay_1y_2+By_1y_3+Cy_2y_3+D=0.
\end{cases}
\end{equation*}
Solving this system for $(y_1y_2,y_1y_3,y_2y_3),$ we obtain
\begin{equation}\label{eq:A3-W12-products}
y_1y_2=-\frac{D}{A},\ y_1y_3=-\frac{D}{B},\ y_2y_3=-\frac{D}{C}.
\end{equation}
Therefore,
\begin{equation}\label{eq:A3-W12-squares}
y_1^2=-\frac{CD}{AB},\ y_2^2=-\frac{BD}{AC},\ y_3^2=-\frac{AD}{BC},
\end{equation}
and, since $\sigma_jy_j>0,$ equation \eqref{eq:A3-W12-z} holds. Since $y_j^2>0,$ we have $ABCD<0.$ A direct computation gives
\begin{equation}\label{eq:A3-ABCD-factorization}
ABCD=-(|\sigma_1|+|\sigma_2|+|\sigma_3|)(-|\sigma_1|+|\sigma_2|+|\sigma_3|)
(|\sigma_1|-|\sigma_2|+|\sigma_3|)(|\sigma_1|+|\sigma_2|-|\sigma_3|).
\end{equation}
Hence,
\[
(-|\sigma_1|+|\sigma_2|+|\sigma_3|)
(|\sigma_1|-|\sigma_2|+|\sigma_3|)
(|\sigma_1|+|\sigma_2|-|\sigma_3|)>0.
\]
We claim that all three factors are positive. Otherwise, since their product is positive, exactly two of them would be negative. By symmetry, we may assume that
\[
-|\sigma_1|+|\sigma_2|+|\sigma_3|<0
\quad\text{and}\quad
|\sigma_1|-|\sigma_2|+|\sigma_3|<0.
\]
Adding these inequalities gives $2|\sigma_3|<0,$ a contradiction. Therefore, all three factors are positive and \eqref{eq:A3-W12-triangle} holds.\\

Conversely, suppose that \eqref{eq:A3-W12-triangle} and \eqref{eq:A3-W12-z} hold. By \eqref{eq:A3-ABCD-factorization}, we have $ABCD<0,$ so the square roots in \eqref{eq:A3-W12-z} are well defined and $\sigma_j\operatorname{Im}(z_j)>0,$ $j=1,2,3.$ A direct substitution of the expressions for $z_j$ into $f_1,f_2,f_3$ shows that $f_1=f_2=f_3=0.$ Hence, $(z_j,\sigma_j)_{j=1}^3$ belongs to $\mathcal{W}_{\{1,2\}}.$\\

Now observe that for each fixed choice of signs
\[
\left(\frac{\sigma_1}{|\sigma_1|},
\frac{\sigma_2}{|\sigma_2|},
\frac{\sigma_3}{|\sigma_3|}\right)\in\{1,-1\}^3,
\]
the corresponding structures are parametrized by the open convex cone
\[
\left\{
(|\sigma_1|,|\sigma_2|,|\sigma_3|)\in(\mathbb{R}^+)^3:
|\sigma_1|<|\sigma_2|+|\sigma_3|,\ 
|\sigma_2|<|\sigma_1|+|\sigma_3|,\ 
|\sigma_3|<|\sigma_1|+|\sigma_2|
\right\}.
\]
A homothety multiplies $(|\sigma_1|,|\sigma_2|,|\sigma_3|)$ by a positive factor. Thus, up to homothety, we may impose $|\sigma_1|+|\sigma_2|+|\sigma_3|=1.$ The resulting space is
\[
\left\{
(|\sigma_1|,|\sigma_2|,|\sigma_3|)\in(\mathbb{R}^+)^3:
|\sigma_1|+|\sigma_2|+|\sigma_3|=1,\ 
|\sigma_j|<\frac{1}{2},\ j=1,2,3
\right\},
\]
which is an open two-dimensional triangle. Since there are eight possible choices of signs, the space of these structures modulo homotheties is a disjoint union of eight manifolds, each of real dimension $2.$
\end{proof}

\begin{proposition}\label{prop:A3-W13}
An invariant almost Hermitian structure $(z_j,\sigma_j)_{j=1}^{3}$ on $\mathbb{F}$ belongs to the Gray-Hervella class $\mathcal{W}_{\{1,3\}}$ if and only if
\begin{equation}\label{eq:A3-W13-condition}
\begin{cases}
(\sigma_1-\sigma_2)z_3(z_1+z_2)=(\sigma_1+\sigma_2)(z_1z_2+1),\\[1mm]
(\sigma_2-\sigma_3)z_1(z_2+z_3)=-(\sigma_2+\sigma_3)(z_2z_3+1).
\end{cases}
\end{equation}
Up to homothety, the set of such structures is a disjoint union of eight manifolds, each of real dimension 4.
\end{proposition}
\begin{proof}
By Proposition~\ref{prop:A3-W4}, $(\nabla\omega)_4=0$ for every $(z_j,\sigma_j)_{j=1}^{3}.$ Therefore, $(z_j,\sigma_j)_{j=1}^{3}$ belongs to $\mathcal{W}_{\{1,3\}}$ if and only if \eqref{eq:A3-W2-condition} holds. A direct computation shows that \eqref{eq:A3-W13-condition} and \eqref{eq:A3-W2-condition} are equivalent. Assume that $(z_j,\sigma_j)_{j=1}^{3}$ satisfies \eqref{eq:A3-W13-condition} and set 
\begin{align*}
    &\tau:=\sigma_1(z_2z_3-z_1z_2+z_1z_3-1)\\
    &=\sigma_2(z_1z_2+z_1z_3+z_2z_3+1)\\
    &=\sigma_3(z_1z_2+z_1z_3-z_2z_3-1),\\
    &p:=z_1z_2,\ q:=z_1z_3,\ r:=z_2z_3,\ \eta:=z_1z_2z_3.
\end{align*}
Then $\eta\neq0$ and
\[
\eta^2=pqr,\ z_1=\frac{pq}{\eta},\ z_2=\frac{pr}{\eta},\ z_3=\frac{qr}{\eta}.
\]
Moreover,
\[
\begin{cases}
    \displaystyle -p+q+r-1=\tau/\sigma_1,\\[2mm]
    \displaystyle p+q+r+1=\tau/\sigma_2,\\[2mm]
    \displaystyle p+q-r-1=\tau/\sigma_3.
\end{cases}
\]
Solving this system, we obtain
\begin{align*}
    p(\tau,\sigma)&=-1+\frac{\tau}{2}\left(\frac{1}{\sigma_2}-\frac{1}{\sigma_1}\right),\\
    q(\tau,\sigma)&=1+\frac{\tau}{2}\left(\frac{1}{\sigma_1}+\frac{1}{\sigma_3}\right),\\
    r(\tau,\sigma)&=-1+\frac{\tau}{2}\left(\frac{1}{\sigma_2}-\frac{1}{\sigma_3}\right),
\end{align*}
where $\sigma:=(\sigma_j)_{j=1}^3.$ Since $\sigma_j\operatorname{Im}(z_j)>0,$ $j=1,2,3,$ we have
\begin{equation}\label{eq:auxiliary-positive-condition}
    \begin{aligned}
    \sigma_1\operatorname{Im}
    \left(\frac{p(\tau,\sigma)q(\tau,\sigma)}{\eta}\right),\ \sigma_2\operatorname{Im}
    \left(\frac{p(\tau,\sigma)r(\tau,\sigma)}{\eta}\right),\ \sigma_3\operatorname{Im}
    \left(\frac{q(\tau,\sigma)r(\tau,\sigma)}{\eta}\right)>0.
    \end{aligned}
\end{equation}
Thus, the assignment $(z_j,\sigma_j)_{j=1}^{3}\longmapsto\left(\tau,\eta,\sigma\right)$ takes the solution set into
    \[
    \mathcal{S}=\left\{
    \left(\tau,\eta,\sigma\right)\in\mathbb{C}\times\mathbb{C}^*\times(\mathbb{R}^*)^3:
    \eta^2=p(\tau,\sigma)q(\tau,\sigma)r(\tau,\sigma)
    \ \text{and \eqref{eq:auxiliary-positive-condition} holds}
    \right\}.
    \]
    
Conversely, each $\left(\tau,\eta,\sigma\right)\in\mathcal{S}$ determines, via 
\[
z_1=\frac{p(\tau,\sigma)q(\tau,\sigma)}{\eta},\ z_2=\frac{p(\tau,\sigma)r(\tau,\sigma)}{\eta},\ z_3=\frac{q(\tau,\sigma)r(\tau,\sigma)}{\eta},
\]
an invariant almost Hermitian structure $(z_j,\sigma_j)_{j=1}^{3}$ satisfying
\eqref{eq:A3-W13-condition}. Therefore, the two assignments above are mutually inverse and define a diffeomorphism between the solution set and $\mathcal{S}$.\\

Now, consider the map $F:\mathbb{C}\times\mathbb{C}^*\times(\mathbb{R}^*)^3
\longrightarrow\mathbb{C}$ defined by
\[
F\left(\tau,\eta,\sigma\right):=\eta^2-p(\tau,\sigma)q(\tau,\sigma)r(\tau,\sigma).
\]
For every $v\in\mathbb{C},$ we have $dF_{\left(\tau,\eta,\sigma\right)}(0,v,0)=2\eta v.$ Since $\eta\neq0,$ this is a surjective real linear map from
$\mathbb{C}$ to $\mathbb{C}.$ Hence, $0$ is a regular value of $F$ and $F^{-1}(0)$ is a smooth submanifold of real codimension $2.$ Since \eqref{eq:auxiliary-positive-condition} consists of strict inequalities, $\mathcal{S}$ is an open subset of $F^{-1}(0).$ Therefore, $\mathcal{S}$ is a smooth manifold of real dimension $5.$ For each
$\epsilon=(\epsilon_j)_{j=1}^3\in\{1,-1\}^3,$ let
\[
\mathcal{S}_{\epsilon}:=\left\{\left(\tau,\eta,\sigma\right)\in\mathcal{S}:\frac{\sigma_j}{|\sigma_j|}=\epsilon_j,\ j=1,2,3\right\}.
\]
Then
\[
\mathcal{S}=\bigsqcup_{\epsilon\in\{1,-1\}^3}\mathcal{S}_{\epsilon},
\]
and each $\mathcal{S}_{\epsilon}$ is both open and closed in $\mathcal{S}.$ Moreover, each of these sets is nonempty since 
\[
(-\epsilon_1\epsilon_2\epsilon_3+\epsilon_2-\epsilon_3-\epsilon_1,-i\epsilon_1\epsilon_2\epsilon_3,(\epsilon_j)_{j=1}^3)\in\mathcal{S}_\epsilon,
\]
for every $\epsilon\in\{1,-1\}^3.$ Consequently, each $\mathcal{S}_\epsilon$ is a smooth manifold of real dimension 5.\\

Finally, a homothety by a factor $\lambda>0$ acts by
\[
\lambda\cdot\left(\tau,\eta,\sigma\right)=\left(\lambda\tau,\eta,(\lambda\sigma_j)_{j=1}^{3}\right).
\]
This action preserves each $\mathcal{S}_{\epsilon}$ and every orbit intersects the level set $|\sigma_1|+|\sigma_2|+|\sigma_3|=1$ at exactly one point, obtained by taking
\[
\lambda=\frac{1}{|\sigma_1|+|\sigma_2|+|\sigma_3|}.
\]
Moreover, the homothety orbits intersect this level set transversely. In fact, along the curve
\[
\lambda\longmapsto
\left(\lambda\tau,\eta,(\lambda\sigma_j)_{j=1}^{3}\right),
\]
we have
\[
\left.
\frac{d}{d\lambda}
\left(
|\lambda\sigma_1|+|\lambda\sigma_2|+|\lambda\sigma_3|
\right)
\right|_{\lambda=1}
=
|\sigma_1|+|\sigma_2|+|\sigma_3|>0.
\]
Thus, this level set is a smooth submanifold of real codimension $1$ in $\mathcal{S}_{\epsilon}$ and is diffeomorphic to the quotient of $\mathcal{S}_{\epsilon}$ by homotheties. Therefore, this quotient has real dimension $4.$ Hence, up to homothety, the solution set is a disjoint union of eight smooth manifolds, each of real dimension $4.$
\end{proof}
\begin{remark}
    The proof of Proposition~\ref{prop:A3-W13} provides a parametrization of the invariant almost Hermitian structures that belong to $\mathcal{W}_{\{1,3\}}$ as follows: let
    \[
    D:=\left\{(\tau,\sigma)\in\mathbb{C}\times(\mathbb{R}^*)^3:\begin{array}{l}\text{there exists }\eta\in\mathbb{C}^*\text{ such that }\eta^2=p(\tau,\sigma)q(\tau,\sigma)r(\tau,\sigma)\\
    \text{and \eqref{eq:auxiliary-positive-condition} holds}
    \end{array}\right\}.
    \]
    For each $(\tau,\sigma)\in D,$ the square root satisfying \eqref{eq:auxiliary-positive-condition} is unique; we denote it by $\eta(\tau,\sigma).$ Moreover, $D$ is open and $\eta:D\to\mathbb{C}^*$ is smooth. Thus, the map
    \[
    (\tau,\sigma)\longmapsto\left(\frac{p(\tau,\sigma)q(\tau,\sigma)}{\eta(\tau,\sigma)},\frac{p(\tau,\sigma)r(\tau,\sigma)}{\eta(\tau,\sigma)},\frac{q(\tau,\sigma)r(\tau,\sigma)}{\eta(\tau,\sigma)},\sigma_1,\sigma_2,\sigma_3\right)
    \]
    is a global smooth parametrization of the set of invariant almost Hermitian structures belonging to $\mathcal{W}_{\{1,3\}}.$
\end{remark}

\begin{proposition}\label{prop:A3-W1}
An invariant almost Hermitian structure $(z_j,\sigma_j)_{j=1}^{3}$ on $\mathbb{F}$ belongs to the Gray-Hervella class $\mathcal{W}_1$ if and only if $(z_j,\sigma_j)_{j=1}^{3}$ is given by one of the rows of Table~\ref{tab:A3-W1}, where $\lambda>0$ and $\delta\in\{1,-1\}.$\\

\noindent
\begin{minipage}{\linewidth}
\centering
\renewcommand{\arraystretch}{1.5}
\setlength{\cellspacetoplimit}{6pt}
\setlength{\cellspacebottomlimit}{6pt}
\setlength{\tabcolsep}{12pt}
\begin{tabular}{|Sc|Sc|}
\hline
$(z_1,z_2,z_3)$
&
$(\sigma_1,\sigma_2,\sigma_3)$
\\
\hline
$\displaystyle\frac{i\delta}{\sqrt{3}}(1,3,1)$
&
$\delta\lambda(1,1,1)$
\\
\hline
$\displaystyle\frac{i\delta}{\sqrt{3}}(3,1,-1)$
&
$\delta\lambda(1,1,-1)$
\\
\hline
$i\delta\sqrt{3}(1,-1,1)$
&
$\delta\lambda(1,-1,1)$
\\
\hline
$\displaystyle\frac{i\delta}{\sqrt{3}}(1,-1,-3)$
&
$\delta\lambda(1,-1,-1)$
\\
\hline
\end{tabular}
\captionof{table}{Invariant nearly K\"ahler structures on the maximal real flag manifold of type $\mathcal{A}_3$.}
\label{tab:A3-W1}
\end{minipage}\\[0.5em]

Consequently, up to homothety, $\mathbb{F}$ admits exactly 
eight invariant nearly K\"ahler structures, all of which are strict.
\end{proposition}
\begin{proof}
By Proposition~\ref{prop:A3-W4}, the $\mathcal{W}_4$-component of
$\nabla\omega$ vanishes. Hence, an invariant almost Hermitian
structure belongs to $\mathcal{W}_1$ if and only if it belongs to both
$\mathcal{W}_{\{1,2\}}$ and $\mathcal{W}_{\{1,3\}}.$ Suppose first that $(z_j,\sigma_j)_{j=1}^3$ belongs to $\mathcal{W}_{\{1,2\}}$ and $\mathcal{W}_{\{1,3\}}.$ By the proof of
Proposition~\ref{prop:A3-W12} we have that $ABCD<0$ and $z_j=iy_j,\ j=1,2,3,$ where $y_1,y_2,y_3$ satisfy \eqref{eq:A3-W12-products}. Consequently, 
\begin{equation*}\label{eq:A3-W1-products}
z_1z_2=\frac{D}{A},\ z_1z_3=\frac{D}{B},\ z_2z_3=\frac{D}{C}.
\end{equation*}
Substituting these expressions into the first equation of \eqref{eq:A3-W13-condition}, we obtain
\begin{align*}
    &\,(\sigma_1-\sigma_2)\left(\frac{D}{B}+\frac{D}{C}\right)=(\sigma_1+\sigma_2)\left(\frac{D}{A}+1\right)\\
    \Longleftrightarrow&\,\sigma_1\left(ACD+ABD-BCD-ABC\right)-\sigma_2\left(ACD+ABD+BCD+ABC\right)=0\\
    \Longleftrightarrow&\,4\sigma_1^2(\sigma_1^2-\sigma_2^2-\sigma_3^2)-4\sigma_2^2(\sigma_2^2-\sigma_1^2-\sigma_3^2)=0\\
    \Longleftrightarrow&\,(\sigma_1^2-\sigma_2^2)(\sigma_1^2+\sigma_2^2-\sigma_3^2)=0.
\end{align*}
Analogously, the second equation of \eqref{eq:A3-W13-condition} reduces to
\[
(\sigma_2^2-\sigma_3^2)(\sigma_1^2-\sigma_2^2-\sigma_3^2)=0.
\]
Therefore,
\begin{equation}\label{eq:A3-W1-sigma-system}
\begin{cases}
(\sigma_1^2-\sigma_2^2)(\sigma_1^2+\sigma_2^2-\sigma_3^2)=0,\\
(\sigma_2^2-\sigma_3^2)
(\sigma_1^2-\sigma_2^2-\sigma_3^2)=0.
\end{cases}
\end{equation}

We claim that $|\sigma_1|=|\sigma_2|=|\sigma_3|.$ In fact, if $\sigma_1^2=\sigma_2^2,$ then the second equation in \eqref{eq:A3-W1-sigma-system} implies either $\sigma_2^2=\sigma_3^2$ or $\sigma_1^2=\sigma_2^2+\sigma_3^2.$ The latter is impossible because $\sigma_1^2=\sigma_2^2$ and $\sigma_3\neq0.$ Hence, $\sigma_2^2=\sigma_3^2.$ On the other hand, suppose that $\sigma_3^2=\sigma_1^2+\sigma_2^2.$ The second equation in \eqref{eq:A3-W1-sigma-system} implies either $\sigma_2^2=\sigma_3^2,$ which would give $\sigma_1=0,$ or $\sigma_1^2=\sigma_2^2+\sigma_3^2,$ which would give $\sigma_2=0.$ Both alternatives are impossible. This proves the claim. Consequently, there exist $\lambda>0$ and $\delta\in\{1,-1\}$ such that $(\sigma_1,\sigma_2,\sigma_3)$ is one of the following triples:
\[
\delta\lambda(1,1,1),\ \delta\lambda(1,1,-1),\ \delta\lambda(1,-1,1),\ \delta\lambda(1,-1,-1).
\]
Substituting these four possibilities into \eqref{eq:A3-W12-z}, we obtain exactly the corresponding values of $(z_1,z_2,z_3)$ displayed in Table~\ref{tab:A3-W1}.\\

Conversely, consider any of the structures of Table~\ref{tab:A3-W1}. Since $|\sigma_1|=|\sigma_2|=|\sigma_3|=\lambda,$ the inequalities \eqref{eq:A3-W12-triangle} hold. Moreover, the corresponding values of $(z_1,z_2,z_3)$ satisfy
\eqref{eq:A3-W12-z}. Thus, the structure belongs to $\mathcal{W}_{\{1,2\}}.$ A direct substitution into \eqref{eq:A3-W13-condition} shows that it also belongs to $\mathcal{W}_{\{1,3\}}.$ Therefore, it belongs to $\mathcal{W}_1.$\\

Finally, up to homothety, we may impose $\lambda=1.$ Since the table contains four rows and $\delta$ has two possible values, there are exactly eight invariant nearly K\"ahler structures up to homothety. By Proposition~\ref{prop:A3-empty-classes}, none of them is K\"ahler. Hence, all of them are strict.
\end{proof}

\begin{remark}
The eight structures in Table~\ref{tab:A3-W1} are equivalent by diffeomorphisms of $\mathbb{F}$ in the sense that, for $\lambda=1$ and any two almost Hermitian structures $(g_j,J_j),\ j=1,2,$ in Table~\ref{tab:A3-W1}, there exists a diffeomorphism $f:\mathbb{F}\to\mathbb{F}$ such that $f^*g_1=g_2$ and $f^*J_1=J_2.$ To see this, write $H:=\mathrm{S}\left(\mathrm{O}(1)\times\mathrm{O}(1)\times\mathrm{O}(1)\times\mathrm{O}(1)\right).$ Given a permutation $\gamma$ of $\{1,2,3,4\},$ let $P_\gamma$ be the $4\times 4$ matrix obtained by applying $\gamma$ to the rows of the identity matrix. Consider the group $D_4=\langle P_{(1234)},P_{(13)}\rangle.$ Since every $P_\gamma\in D_4$ is orthogonal and normalizes $H,$ conjugation by $P_\gamma$ induces a diffeomorphism
\[
C_{P_\gamma}:\mathbb{F}\longrightarrow\mathbb{F},\ C_{P_\gamma}(xH)=P_{\gamma}xP_{\gamma}^{-1}H,\ (x\in\mathrm{SO}(4)).
\]
Let $(g_{r,\delta},J_{r,\delta})$ denote the structure determined by the $r$-th row of Table~\ref{tab:A3-W1}, with $\lambda=1$ and $\delta\in\{1,-1\}.$ Then, one verifies that
\[
C_{P_\gamma}^*g_{1,1}=g_{r,\delta},\ C_{P_{\gamma}}^*J_{1,1}=J_{r,\delta}
\]
for the following choices:
\[
\begin{array}{|c|c|c|c|c|c|c|c|c|}
\hline(r,\delta)&(1,1)&(1,-1)&(2,1)&(2,-1)&(3,1)&(3,-1)&(4,1)&(4,-1)\\\hline
\gamma&\operatorname{Id}&(13)&(1432)&(14)(23)&(24)&(13)(24)&(12)(34)&(1234)\\\hline
\end{array}
\]
Thus, the eight structures are equivalent.\\

This is consistent with Butruille's classification of simply connected six-dimensional homogeneous strict nearly K\"ahler manifolds in \cite[Theorem~1.2]{B}. In fact, let $\rho:\operatorname{Spin}(4)\longrightarrow\operatorname{SO}(4)$ be the double covering. Since $H$ is discrete, the map $\operatorname{Spin}(4)\ni a\longmapsto\rho(a)H\in\mathbb{F}$ is the universal covering of $\mathbb{F}.$ Therefore,
\[
\tilde{\mathbb{F}}\simeq\operatorname{Spin}(4)\simeq\operatorname{SU}(2)\times\operatorname{SU}(2)\simeq S^3\times S^3
\]
and the unique (up to equivalence) structure in Table~\ref{tab:A3-W1} lifts to the left-invariant strict nearly K\"ahler structure on $S^3\times S^3$ described in \cite[Proposition~3.1]{B}.
\end{remark}

\begin{proposition}
    An invariant almost Hermitian structure $(z_j,\sigma_j)_{j=1}^{3}$ on $\mathbb{F}$ belongs to the Gray-Hervella class $\mathcal{W}_{\{2,3\}}$ if and only if
    \begin{equation}\label{eq:A3-W23-condition}
    Az_1z_2+Bz_1z_3+Cz_2z_3+D=0.
    \end{equation}
    Up to homothety, the set of such structures is a disjoint union of eight manifolds, each of real dimension $6.$
\end{proposition}
\begin{proof}
    By Proposition~\ref{prop:A3-W4}, $(\nabla\omega)_4=0$ for every invariant almost Hermitian structure. Therefore, $(z_j,\sigma_j)_{j=1}^{3}$ belongs to $\mathcal{W}_{\{2,3\}}$ if and only if $(\nabla\omega)_1=0.$ By \eqref{eq:A3-W1-condition}, this is equivalent to \eqref{eq:A3-W23-condition}. Consider the smooth map $F:\mathcal{Z}\longrightarrow\mathbb{C}$ defined by
    \[
    F((z_j,\sigma_j)_{j=1}^3):=Az_1z_2+Bz_1z_3+Cz_2z_3+D.
    \]
    Its complex partial derivatives with respect to $z_1,z_2,z_3$ are
    \[
    \frac{\partial F}{\partial z_1}=Az_2+Bz_3,\ \frac{\partial F}{\partial z_2}=Az_1+Cz_3,\ \frac{\partial F}{\partial z_3}=Bz_1+Cz_2.
    \]
    An argument analogous to the one used in the proof of Proposition~\ref{prop:A3-W12} shows that these three complex partial derivatives cannot be zero simultaneously. Therefore, the real differential of $F$ is surjective and $0$ is a regular value of $F.$ Since $\mathcal{Z}$ has real dimension $9,$  $F^{-1}(0)$ is an embedded submanifold of real dimension 7. For each $\epsilon=(\epsilon_j)_{j=1}^{3}\in\{1,-1\}^3,$ let
    \[
    \mathcal{S}_{\epsilon}:=\left\{(z_j,\sigma_j)_{j=1}^{3}\in F^{-1}(0):\frac{\sigma_j}{|\sigma_j|}=\epsilon_j,\ j=1,2,3\right\}.
    \]
    Then 
    \[
    F^{-1}(0)=\bigsqcup_{\epsilon\in\{1,-1\}^3}\mathcal{S}_\epsilon
    \]
    and each $\mathcal{S}_\epsilon$ is both closed and open in $F^{-1}(0).$ Moreover, each $\mathcal{S}_\epsilon$ is nonempty because 
    \[
    (i\epsilon_1,i\epsilon_23^{\epsilon_1\epsilon_3},i\epsilon_33^{-\epsilon_1\epsilon_2},10\epsilon_1,\epsilon_2,\epsilon_3)\in\mathcal{S}_\epsilon. 
    \]
    Finally, a homothety by a factor $\lambda>0$ acts on the parameters by
    \[
    \lambda\cdot(z_j,\sigma_j)_{j=1}^3=(z_j,\lambda\sigma_j)_{j=1}^3.
    \]
    Since $F((z_j,\lambda\sigma_j)_{j=1}^3)=\lambda F((z_j,\sigma_j)_{j=1}^3),$ this action preserves each $\mathcal{S}_\epsilon.$ Every homothety orbit intersects transversely the level set $|\sigma_1|+|\sigma_2|+|\sigma_3|=1$ at exactly one point. Hence, this level set is a smooth submanifold of real codimension $1$ in each $\mathcal{S}_\epsilon$ and is diffeomorphic to its quotient by homotheties. Therefore, each quotient has real dimension $6$ and the set $F^{-1}(0)$, up to homothety, is a disjoint union of eight smooth manifolds, each of real dimension $6.$
\end{proof}
\begin{remark}
The set of invariant almost Hermitian structures on $\mathbb{F}$ that belong to the Gray-Hervella class $\mathcal{W}_{\{2,3\}}$ admits a global parametrization in terms of $(z_1,z_2,(\sigma_j)_{j=1}^3).$ To obtain this parametrization, we shall show that $\partial F/\partial z_3\neq0$ for every $(z_j,\sigma_j)_{j=1}^3\in F^{-1}(0).$ Assume that $Bz_1+Cz_2=0$ and that $F((z_j,\sigma_j)_{j=1}^3)=0.$ Since $B-C=2\sigma_3\neq 0$ and $z_1,z_2\neq 0,$ we have $BC\neq 0$ and $z_2=-(B/C)z_1.$ Now, $F((z_j,\sigma_j)_{j=1}^3)=0$ implies $-Az_1^2(B/C)+D=0.$ Since $A-D=2\sigma_3\neq0$ and $z_1\neq 0,$ we have $AD\neq0$ and $z_1^2=CD/AB.$ The right-hand side is real. Since $\operatorname{Im}(z_1)\neq0,$ it follows that $\operatorname{Re}(z_1)=0.$ Consequently, $\operatorname{Re}(z_2)=0$ as well. Write $z_1=iy_1,$ $z_2=iy_2,$ where $\sigma_1y_1>0$ and $\sigma_2y_2>0.$ Taking imaginary parts in $Bz_1+Cz_2=0$ we obtain $By_1+Cy_2=0.$ Therefore, $BCy_1y_2<0.$ On the other hand, since $z_1=iy_1,$ we have $z_1^2=-y_1^2<0.$ Together with $z_1^2=CD/AB,$ this implies that $ABCD<0.$ By the proof of Proposition~\ref{prop:A3-W12}, the numbers $|\sigma_1|,|\sigma_2|,|\sigma_3|$ must satisfy the strict triangle inequalities. We claim that $BC\sigma_1\sigma_2>0.$ In fact, observe that $BC=(\sigma_1+\sigma_2)^2-\sigma_3^2.$ If $\sigma_1\sigma_2>0,$ then $|\sigma_1+\sigma_2|=|\sigma_1|+|\sigma_2|>|\sigma_3|$ and hence $BC>0.$ If $\sigma_1\sigma_2<0,$ then $|\sigma_1+\sigma_2|=\bigl||\sigma_1|-|\sigma_2|\bigr|<|\sigma_3|$ and hence $BC<0.$ Thus, $BC\sigma_1\sigma_2>0$ in both cases. Since $\sigma_1y_1>0$ and $\sigma_2y_2>0,$ the numbers $y_1y_2$ and $\sigma_1\sigma_2$ have the same sign. Hence, $BC\sigma_1\sigma_2>0$ implies $BCy_1y_2>0,$ contradicting
$BCy_1y_2<0.$
    Therefore,
    \[
        \frac{\partial F}{\partial z_3}=Bz_1+Cz_2\neq0
    \]
    at every point of $F^{-1}(0).$ Consequently, every element of $F^{-1}(0)$ is uniquely determined by $(z_1,z_2,(\sigma_j)_{j=1}^3)$ through
    \begin{equation*}
        z_3=-\frac{Az_1z_2+D}{Bz_1+Cz_2}.
    \end{equation*}
    Let
    \[
    \begin{aligned}
    U:=\biggl\{(z_1,z_2,(\sigma_j)_{j=1}^3)\in\mathbb{C}^2\times(\mathbb{R}^*)^3:&\,Bz_1+Cz_2\neq0,\ \sigma_j\operatorname{Im}(z_j)>0,\ j=1,2,\\
    &\sigma_3\operatorname{Im}\left(-\frac{Az_1z_2+D}{Bz_1+Cz_2}\right)>0\biggr\}.
    \end{aligned}
    \]
    Then the map
    \[
    U\ni (z_1,z_2,(\sigma_j)_{j=1}^3)\longmapsto\left(z_1,z_2,-\frac{Az_1z_2+D}{Bz_1+Cz_2},(\sigma_j)_{j=1}^3\right)
    \]
is a global parametrization of the set of invariant almost Hermitian structures on $\mathbb{F}$ that belong to $\mathcal{W}_{\{2,3\}}.$
\end{remark}

\begin{example}
    Consider the invariant almost Hermitian structure $(g,J)$, where 
    \[
J_{j}=\begin{pmatrix}
0 & 1\\
-1 & 0
\end{pmatrix},
\ 
A_{j}=\begin{pmatrix}
1 & 0\\
0 & 1
\end{pmatrix}, j=1,2,3.
\]
In this case, $(g,J)$ is in the class $\mathcal{W}_{\{1,3\}}$.
\end{example}

\begin{example}
    Consider the invariant almost Hermitian structure $(g,J)$, where 
    \[
J_{1}=\begin{pmatrix}
0 & 1\\
-1 & 0
\end{pmatrix},
\ 
J_{2}=\begin{pmatrix}
0 & 3\\
-1/3 & 0
\end{pmatrix},
\
J_{3}=\begin{pmatrix}
0 & 1/3\\
-3 & 0
\end{pmatrix}
\]
\[
A_{1}=\begin{pmatrix}
10 & 0\\
0 & 10
\end{pmatrix},
\ 
A_{2}=\begin{pmatrix}
1/3 & 0\\
0 & 3
\end{pmatrix},
\
A_{3}=\begin{pmatrix}
3 & 0\\
0 & 1/3
\end{pmatrix}
\]
In this case, $(g,J)$ is in the class $\mathcal{W}_{\{2,3\}}$.
\end{example}

\begin{example}
    Consider the invariant almost Hermitian structure $(g,J)$, where 
    \[
J_{j}=\begin{pmatrix}
1 & 2\\
-1 & -1
\end{pmatrix},
\ 
A_{j}=\begin{pmatrix}
1 & 1\\
1 & 2
\end{pmatrix}, j=1,2,3.
\]
In this case, $(g,J)$ is in the class $\mathcal{W}_{\{1,2,3\}}$.
\end{example}

\bibliographystyle{apa}

\begin{thebibliography}{99}
\bibitem{AGS} Abbena, E., Garbiero, S. and Salamon, S. Almost Hermitian geometry on six dimensional nilmanifolds. {\it Ann. Scuola Norm. Sup. Pisa Cl. Sci. (4)} \textbf{30}(1) (2001): 147--170.

\bibitem{AKW} Alekseevsky, D. V., Kruglikov, B. S. and Winther, H. Homogeneous almost complex structures in dimension 6 with semi-simple isotropy. {\it Ann. Global Anal. Geom.} \textbf{46}(4) (2014): 361--387.

\bibitem{AS} Alves, L. A. and da Silva, N. P. Invariant $\mathcal{G}_1$ structures on flag manifolds. {\it Geom. Dedicata} \textbf{213} (2021): 227--243.

\bibitem{AT} Andrada, A. and Tolcachier, A. Harmonic almost complex structures on almost abelian Lie groups and solvmanifolds. {\it Ann. Mat. Pura Appl.} \textbf{203}(3) (2024): 1037--1060.

\bibitem{B} Butruille, J.-B. Classification des variétés approximativement kähleriennes homogènes. {\it Ann. Global Anal. Geom.} \textbf{27} (2005): 201--225.

\bibitem{CS} Cabrera, F. M. and Swann, A. The intrinsic torsion of almost quaternion-Hermitian manifolds. \textit{Ann. Inst. Fourier (Grenoble)} {\bf 58}(5) (2008): 1455--1497.

\bibitem{FP} Fino, A. and Paradiso, F. Balanced Hermitian structures on almost abelian Lie algebras. {\it J. Pure Appl. Algebra} \textbf{227}(2) (2023): 107186.

\bibitem{FBSM} Freitas, A. P. C., del Barco, V. and San Martin, L. A. B. Invariant almost complex structures on real flag manifolds. {\it Ann. Mat. Pura Appl.} {\bf 197} (2018): 1821--1844.

\bibitem{G} Gauduchon, P. Hermitian connections and Dirac operators. {\it Boll. Un. Mat. Ital. B} {\bf 11}(2) (1997): 257--288.

\bibitem{GO} Grama, L. and Oliveira, A. R. Scalar curvatures of invariant almost Hermitian structures on generalized flag manifolds. {\it SIGMA Symmetry Integrability Geom. Methods Appl.} {\bf 17} (2021): 109.

\bibitem{GGN} Grajales, B., Grama, L. and Negreiros, C. J. C. Geodesic orbit spaces in real flag manifolds. {\it Comm. Anal. Geom.} {\bf 28}(8) (2020): 1933--2003.

\bibitem{GH} Gray, A. and Hervella, L. M. The sixteen classes of almost Hermitian manifolds and their linear invariants. \textit{Ann. Mat. Pura Appl.} {\bf 123} (1980): 35--58.

%\bibitem{Michelsohn} Michelsohn, M.-L. On the existence of special metrics in complex geometry. {\it Acta Math.} {\bf 149}(3--4) (1982): 261--295.

\bibitem{NN} Newlander, A. and Nirenberg, L. Complex analytic coordinates in almost complex manifolds. {\it Ann. of Math.} \textbf{65}(3) (1957): 391--404.

\bibitem{N} Nomizu, K. Invariant affine connections on homogeneous spaces. {\it Amer. J. Math.} {\bf 76}(1) (1954): 33--65.

\bibitem{SMN} San Martin, L. A. B. and Negreiros, C. J. C. Invariant almost Hermitian structures on flag manifolds. {\it Adv. Math.} \textbf{178}(2) (2003): 277--310.

\bibitem{SMS} San Martin, L. A. B. and Silva, R. C. J. Invariant nearly-K\"ahler structures. {\it Geom. Dedicata} \textbf{121} (2006): 143--154.

%\bibitem{PSM} Patr{\~a}o, M. and San Martin, L. A. B. The isotropy representation of a real flag manifold: Split real forms. {\it Indag. Math.} {\bf 26}(3) (2015): 547--579.

%\bibitem{Vaisman} Vaisman, I. On locally conformal almost Kähler manifolds. {\it Israel J. Math.} {\bf 24}(3--4) (1976): 338--351.
\end{thebibliography}

\end{document}